\documentclass[english,11pt,leqno]{amsart}

\usepackage{amsmath,amsfonts,amssymb,amscd,amsthm,amsbsy,upref, mathabx}
\usepackage{color}
 \usepackage{comment}
 \usepackage{amsrefs}
 \usepackage[english]{babel}
 \usepackage{xspace}

\newcommand{\N}{\mathbb{N}}
\newcommand{\R}{\mathbb{R}}

\newcommand{\M}{\mathbb{M}}

\newcommand{\A}{\mathcal{A}}

\newcommand{\ie}{\textit{i.e }}

\newcommand{\Lip}{\operatorname{Lip}}

\newcommand{\Tr}{\operatorname{Tr}}
\newcommand{\SB}{\operatorname{SB}}
\newcommand{\dH}{d_{\mathbb{H}}}

\newcommand{\nb}{\overline{n}}
\newcommand{\mb}{\overline{m}}

\newcommand{\vspan}{\operatorname{span}}
\newcommand{\ee}{\varepsilon}

\theoremstyle{plain}
\newtheorem{prop}{Proposition}[section]
\newtheorem{theorem}[prop]{Theorem}
\newtheorem{lemma}[prop]{Lemma}
\newtheorem{corollary}[prop]{Corollary}
\theoremstyle{definition}
\newtheorem{definition}[prop]{Definition}

\newtheorem{pb}{Problem}
\newtheorem*{thank}{Acknowledgments}
\theoremstyle{remark}
\newtheorem{rmk}[prop]{Remark}

\begin{document}

\allowdisplaybreaks

\title[Asymptotic smoothness, concentration properties and applications]{Asymptotic smoothness, concentration properties in Banach spaces and applications}

\author{A.~Fovelle}
\address{Institute of Mathematics (IMAG) and Department of Mathematical Analysis, University of Granada, 18071, Granada, Spain}
\email{audrey.fovelle@ugr.es}

\thanks{Research partially supported by MCIN/AEI/10.13039/501100011033 grant PID2021-122126NB-C31 and by ``Maria de Maeztu'' Excellence Unit IMAG, reference CEX2020-001105-M funded by MCIN/AEI/10.13039/501100011033}

\keywords{Banach spaces, Hamming graphs, Asymptotic smoothness, nonlinear embeddings, concentration properties, Szlenk index}

\begin{abstract} We prove an optimal result of stability under $\ell_p$-sums of some concentration properties for Lipschitz maps defined on Hamming graphs into Banach spaces. As an application, we give examples of spaces with Szlenk index arbitrarily high that admit nevertheless a concentration property. In particular, we get the very first examples of Banach spaces with concentration but without asymptotic smoothness property.
\end{abstract}

\maketitle

 \setcounter{tocdepth}{1}

\section{Introduction}

In 2008, in order to show that $L_p(0,1)$ is not uniformly homeomorphic to $\ell_p \oplus \ell_2$ for $p \in (1, \infty) \setminus \{2 \}$, Kalton and Randrianarivony \cite{KR} introduced a new technique based on a certain class of graphs and asymptotic smoothness ideas. More specifically, they proved that reflexive asymptotically uniformly smooth (AUS) Banach spaces (we refer the reader to Section $2$ for the definitions) have a concentration property for Lipschitz maps defined on Hamming graphs. This concentration property is stable under coarse Lipschitz embeddings (\ie under maps that are bi-Lipschitz for large distances, see Section 2 for a precise definition) and prevents the equi-Lipschitz embeddings of the Hamming graphs, which makes it an obstruction to the coarse Lipschitz embedding of spaces without (enough) concentration. Their result was later extended to the quasi-reflexive case by Lancien et Raja \cite{LR}, who introduced a weaker concentration property that still provides an obstruction to coarse Lipschitz embeddings. Soon after, Causey \cite{Causey3.5} proved that this same weaker concentration property also applies to quasi-reflexive spaces with so-called upper $\ell_p$ tree estimates. The proof of these results is in two steps. In the first one, vaguely speaking, reflexivity or quasi-reflexivity is used to see a Lipschitz map defined on the Hamming graphs as the sum of branches of a weakly-null tree. The second one consists in dealing with this weakly-null tree using a hypothesis of upper $\ell_p$ tree estimates. So as to get examples of non-quasi-reflexive spaces with concentration properties, the author started a general study of these concentration properties, together with new ones and proved that $\ell_p$-sums of quasi-reflexive spaces with these so-called upper $\ell_p$ tree estimates admit concentration. To do so, the following general theorem was proved \cite{fov}:

\begin{theorem} \label{thm somme E}
Let $p \in (1, \infty)$, $\lambda > 0$, $(X_n)_{n \in \N}$ a sequence of Banach spaces with property $\lambda$-HFC$_{p,d}$ (resp. $\lambda$-HIC$_{p,d}$). \\
Then $\left( \sum_{n \in \N} X_n \right)_{\ell_p}$ has property $(\lambda+2+\varepsilon)$-HFC$_{p,d}$ (resp. $(\lambda+2+\varepsilon)$-HIC$_{p,d}$) for every $\varepsilon > 0$.
\end{theorem}

where property $\lambda$-HFC$_{p,d}$ (resp. $\lambda$-HIC$_{p,d}$), defined in \cite{fov}, is a refinement of the property $\lambda$-HFC$_p$ (resp. $\lambda$-HIC$_p$) first considered by Kalton and Randrianarivony (resp. Lancien and Raja): a space $X$ has property $\lambda$-HFC$_p$ (resp. $\lambda$-HIC$_p$) if for any Lipschitz function $f : ([\N]^k, \dH) \to X$, there exist $\nb < \mb$ (resp. $n_1 < m_1 < n_2 < m_2 < \ldots < n_k < m_k$) such that $\|f(\nb)-f(\mb)\| \leq \lambda k^{1/p} \Lip(f)$ (see Section 2 for the definitions of the concentration properties).  

Since the only known examples of reflexive or quasi-reflexive spaces with concentration were those with an asymptotic smoothness property (given by the so-called upper $\ell_p$ tree estimates), the new examples of spaces with concentration given by Theorem \ref{thm somme E} also happen to have an asymptotic smoothness property. However, as mentioned in \cite{fov}, if one manages to prove Theorem \ref{thm somme E} without a loss on the constant, one would get an example of a Banach space admitting concentration without an asymptotic smoothness property. This is what we are going to do in this paper. As a consequence, we will even get Banach spaces with Szlenk index as big as we want that cannot equi-Lipschitz contain the Hamming graphs. This is the very first result of this type for a non-asymptotically smooth Banach space. In order to show this result, we introduce the notions and the terminology we will use in the second section of this paper while Section $3$ is dedicated to the proof of the main result:

\begin{theorem} \label{thm somme}
Let $\lambda \geq 2$ and $(X_n)$ be a sequence of Banach spaces with property $\lambda$-HFC$_{p,d}$ (resp. property $\lambda$-HIC$_{p,d}$).
For every $\varepsilon>0$, $X=(\sum_n X_n)_{\ell_p}$ has property $(\lambda+\varepsilon)$-HFC$_{p,d}$ (resp. property $(\lambda+\varepsilon)$-HIC$_{p,d}$).
\end{theorem}

In section $4$, we detail how this theorem allows us to get Banach spaces with Szlenk index as big as we want with some concentration and we study the topological complexity of the class of Banach spaces having HFC$_{p,d}$.

\section{Definitions and notation}

All Banach spaces in these notes are assumed to be real and infinite-dimensional unless otherwise stated. We denote the closed unit ball of a Banach space $X$ by $B_X$, and its unit sphere by $S_X$. Given a Banach space $X$ with norm $\|\cdot\|_X$, we simply write $\|\cdot\|$ as long as it is clear from the context on which space it is defined. \\
Let $(X_n)_{n \in \N}$ be a sequence of Banach spaces and $p \in [1, \infty)$. We define the sum $\left( \sum_{n \in \N} X_n \right)_{\ell_p}$ to be the space of sequences $(x_n)_{n \in \N}$, where $x_n \in X_n$ for all $n \in \N$, such that $\sum_{n \in \N} \|x_n\|_{X_n}^p$ is finite, and we set
\[ \|(x_n)_{n \in \N} \| = \Big( \sum_{n \in \N} \|x_n\|_{X_n}^p \Big)^{\frac{1}{p}} . \]
One can check that $\left( \sum_{n \in \N} X_n \right)_{\ell_p}$, endowed with the norm $\|\cdot\|$ defined above, is a Banach space. We can, in a similar way, define finite sums $\left( \sum_{j=1}^n X_j \right)_{\ell_p}$ for all $n \in \N$, and, in case $n=2$, we will write $X_1 \bigoplus\limits_p X_2$. 

\subsection{Hamming graphs}

Before introducing concentration properties, we need to define special metric graphs that we shall call \textit{Hamming graphs}. Let $\M$ be an infinite subset of $\N$. We denote by $[\M]^{\omega}$ the set of infinite subsets of $\M$. For $\M \in [\N]^{\omega}$ and $k \in \N$, let
\[ [\M]^k = \{ \overline{n}=(n_1, \dots, n_k) \in \M^k ; n_1 < \cdots < n_k \} , \]
\[ [\M]^{\leq k} = \bigcup_{j=1}^k [\M]^j \cup \{ \varnothing \} , \]
and
\[ [\M]^{< \omega} = \bigcup\limits_{k=1}^{\infty} [\M]^k \cup \{ \varnothing \} . \]
Then we equip $[\M]^k$ with the \textit{Hamming distance}:
\[ \dH(\overline{n}, \overline{m})=| \{ j ; n_j \neq m_j \} |  \]
for all $\overline{n}=(n_1, \dots, n_k), \overline{m}=(m_1, \dots, m_k) \in [\M]^k$. Let us mention that this distance can be extended to $[\M]^{< \omega}$ by letting
\[ \dH(\overline{n}, \overline{m})=| \{ i \in \{1, \cdots, \min(l,j) \} ; n_i \neq m_i \} | + |l-j|  \]
for all $\overline{n}=(n_1, \dots, n_l), \overline{m}=(m_1, \dots, m_j) \in [\M]^{< \omega}$ (with possibly $l=0$ or $j=0$). We also need to introduce $I_k(\M)$, the set of strictly interlaced pairs in $[\M]^k$:
\[ I_k(\M)= \{ (\overline{n}, \overline{m}) \in [\M]^k \times [\M]^k ; n_1 < m_1 < \cdots < n_k < m_k  \} \]
and, for each $j \in \{ 1, \cdots, k \}$, let
\[ H_j(\M)= \{ (\overline{n}, \overline{m}) \in [\M]^k \times [\M]^k ; \forall i \neq j, n_i=m_i \text{ and } n_j < m_j \} . \]

Let us mention that, in this paper, we will only be interested in the Hamming distance but originally, when Hamming graphs were used in \cite{KR}, it could be equally replaced (unless for their last Theorem 6.1) by the \textit{symmetric distance}, defined by
\[ d_{\Delta}(\overline{n}, \overline{m})=\dfrac{1}{2} | \overline{n} \triangle \overline{m} | \] 
for all $\overline{n}, \overline{m} \in [\N]^{< \omega}$, where $\overline{n} \triangle \overline{m}$ denotes the symmetric difference between $\overline{n}$ and $\overline{m}$.

\subsection{Asymptotic uniform smoothness and Szlenk index}

Let us start this subsection by defining the asymptotic uniform smoothness. Let $(X, \|\cdot\|)$ be a Banach space. Following Milman (see \cite{milman}), we introduce the following modulus: for all $t \geq 0$, let
\[ \overline{\rho}_X(t)= \sup\limits_{x \in S_X} \inf\limits_{Y} \sup\limits_{y \in S_Y} (\|x+ty\|-1) \]
where $Y$ runs through all closed linear subspaces of $X$ of finite codimension. We say that $\|\cdot\|$ is \textit{asymptotically uniformly smooth} (in short AUS) if $\lim_{t \to 0} \frac{\overline{\rho}_X(t)}{t}=0$. If $p \in (1, \infty )$, $\|\cdot\|$ is said to be \textit{$p$-AUS} if there is a constant $C > 0$ such that, for all $t \in [0, \infty )$, $\overline{\rho}_X(t) \leq C t^p$. If $X$ has an equivalent norm for which $X$ is AUS (resp. $p$-AUS), $X$ is said to be \textit{AUSable} (resp. \textit{$p$-AUSable}). This notion is related to the Szlenk index which definition we will recall now. It is based on the following Szlenk derivation. For a Banach space $X$, $K\subset X^*$ weak$^*$-compact, and $\ee>0$, we let $s_\ee(K)$ denote the set of $x^*\in K$ such that for each weak$^*$-neighborhood $V$ of $x^*$, $\text{diam}(V\cap K)\ge \ee$. Then we  define the transfinite derivations as follows
	\[s_\ee^0(K)=K,\] \[s^{\xi+1}_\ee(K)=s_\ee(s^\xi_\ee(K)),\] for every ordinal $\xi$ and if $\xi$ is a limit ordinal, \[s^\xi_\ee(K)=\bigcap_{\zeta<\xi}s_\ee^\zeta(K).\]  
	For convenience, we let $s_0(K)=K$. If there exists an ordinal $\xi$ such that $s^\xi_\ee(K)=\varnothing$, we let $Sz(K,\ee)$ denote the minimum such ordinal, and otherwise we write $Sz(K,\ee)=\infty$.   We let $Sz(K)=\sup_{\ee>0} Sz(K,\ee)$, where $Sz(K)=\infty$ if $Sz(K,\ee)=\infty$ for some $\ee>0$. We let $Sz(X,\ee)=Sz(B_{X^*},\ee)$ and $Sz(X)=Sz(B_{X^*})$. If $X$ is separable, $Sz(X)$ is equal to the index associated with the following derivation:
\[ l_\varepsilon(K)=\{ x^* \in K; \exists (x_n^*) \subset K, \ \forall n \in \N, \ \|x_n^*-x^*\|\geq \varepsilon, \ x_n^* \overset{\omega^*}{\longrightarrow} x^* \} \]
if $K\subset X^*$ weak$^*$-compact, and $\ee>0$.

Let us remark that $Sz(X)< \infty$ if and only if $X$ is \textit{Asplund}, that is to say every separable subspace of $X$ has separable dual. For more information on the Szlenk index, we refer the reader to the survey \cite{Lancien2006}. To conclude this subsection, let us remind the following fondamental renorming result, due to Knaust, Odell and Schlumprecht \cite{knaust1999asymptotic} in the separable case and to Raja \cite{raja} in the non separable one (see \cite{GKL} for the precise quantitative version): 

\begin{theorem} \label{thm Sz}
Let $X$ be an infinite dimensional Banach space. Then $Sz(X)=\omega$ if and only if $X$ is AUSable if and only if $X$ is $p$-AUSable for some $p \in (1, \infty)$.
\end{theorem}

\subsection{Lipschitz and coarse Lipschitz embeddings}

Let us recall some definitions on metric embeddings. Let $(X,d_X)$ and $(Y, d_Y)$ two metric spaces, $f$ a map from $X$ to $Y$. We define the \textit{compression modulus} of $f$ by 
\[ \forall t \geq 0, \hspace{2mm} \rho_f(t)= \inf \{ d_Y(f(x),f(y)) ; d_X(x,y) \geq t \} ; \]
and the \textit{expansion modulus} of $f$ by 
\[ \forall t \geq 0, \hspace{0.2cm} \omega_f(t)= \sup \{ d_Y(f(x),f(y)) ; d_X(x,y) \leq t \} . \]
We adopt the convention $\inf(\varnothing)= + \infty$. Note that, for every $x,y \in X$, we have 
\[ \rho_f(d_X(x,y)) \leq d_Y(f(x),f(y)) \leq \omega_f(d_X(x,y)) . \]
If one is given a family of metric spaces $(X_i)_{i \in I}$, one says that $(X_i)_{i \in I}$ \textit{equi-Lipschitz embeds} into $Y$ if there exist $A, B$ in $(0, \infty )$ and, for all $i \in I$, maps $f_i : X_i \to Y$ such that $\rho_{f_i}(t) \geq At$ and $\omega_{f_i}(t) \leq Bt$ for all $t \geq 0$. And we say that $f$ is a \textit{coarse Lipschitz embedding} if there exist $A, B, C, D$ in $(0, \infty)$ such that $\rho_f(t) \geq At-C$ and $\omega_f(t) \leq Bt+D$ for all $t \geq 0$. If $X$ and $Y$ are Banach spaces, this is equivalent to the existence of numbers $\theta \geq 0$ and $0 < c_1 < c_2$ so that :
\[ c_1 \|x-y\|_X \leq \|f(x)-f(y)\|_Y \leq c_2 \|x-y\|_X \]
for all $x,y \in X$ satisfying $\|x-y\|_X \geq \theta$.  

\subsection{Definitions of concentration properties}

Let us start this subsection by recalling a version of Ramsey's theorem we will use.

\begin{theorem}[Ramsey's theorem \cite{Ramsey}] \label{ramsey}
Let $k \in \N$ and $\A \subset [\N]^k$. \\
There exists $\M \in [\N]^{\omega}$ such that either $[\M]^k \subset \A$ or $[\M]^k \cap \A = \varnothing$.
\end{theorem}


We now introduce the two concentration properties we will study here, defined in \cite{fov}. The first one, called ``Hamming Full Concentration'' asks for some concentration condition for all pairs of elements of a subgraph, when the second one, called ``Hamming Interlaced Concentration'' only requires it for pairs of interlaced elements. The ``$d$'' in subscript in the acronyms below can be explained by the fact that sort of directional Lipschitz constants take part.

\begin{definition}
Let $(X,d)$ be a metric space, $\lambda > 0$, $p \in (1, \infty)$.  We say that $X$ has property \textit{$\lambda$-HFC$_{p,d}$} (resp. \textit{$\lambda$-HIC$_{p,d}$}) if, for every $k \in \N$ and every bounded function $f : [\N]^k \to X$, there exists $\M \in [\N]^{\omega}$ such that
\[ d(f(\overline{n}),f(\overline{m}) ) \leq \lambda \left( \sum_{j=1}^k \Lip_j(f)^p \right)^{\frac{1}{p}} \]
for all $\overline{n}, \overline{m} \in [\M]^k$ (resp. $(\overline{n}, \overline{m}) \in I_k(\M)$), where, for each $j \in \{1, \cdots, k\}$
\[ \Lip_j(f) = \sup\limits_{(\overline{n}, \overline{m}) \in H_j(\N)} d(f(\overline{n}),f(\overline{m})) . \]
We say that $X$ has property \textit{HFC$_{p,d}$} (resp. \textit{HIC$_{p,d}$}) if $X$ has property $\lambda$-HFC$_{p,d}$ (resp. $\lambda$-HIC$_{p,d}$), for some $\lambda > 0$. 
\end{definition}

It is important to note that Theorem 6.1 \cite{KR} and Theorem 2.4 \cite{LR} can be rephrased as follows: for $p \in (1, \infty)$, a reflexive (resp. quasi-reflexive) $p$-AUSable Banach space has property HFC$_{p,d}$ (resp. HIC$_{p,d}$). As mentionned by Causey \cite{Causey3.5}, these results still apply when $X$ is only assumed to have upper $\ell_p$ tree estimates (more precisely when $X$ has property $A_p$, defined for example in \cite{Causey3.5}). In order to get non-quasi-reflexive spaces with some concentration, the author proved that these properties are stable under $\ell_p$ sums. It is worth mentioning that these properties are stable under coarse Lipschitz embeddings when the embedded space is a Banach space and they clearly prevent the equi-Lipschitz embeddings of the Hamming graphs. As mentioned in \cite{fov}, they also prevent the equi-Lipschitz embeddings of the symmetric graphs.

\section{Stability under sums} \label{section sums}

In order to prove the stability of properties HFC$_{p,d}$ and HIC$_{p,d}$, $p \in (1, \infty)$, under $\ell_p$ sums, with the same constant, we will improve the idea from \cite{fov}. To do so, we need the following proposition, whose proof can be deduced from \cite{fov} but that we include for sake of completeness.

\begin{lemma}
Let $p \in (1, + \infty)$, $X_1$ and $X_2$ two Banach spaces, $X=X_1 \bigoplus\limits_p X_2$, $k \in \N$. \\
For every $\varepsilon>0$ and every Lipschitz map $h=(f,g) : [\N]^k \to X$, there exists $\M \in [\N]^\omega$ such that
\[ \Lip_j(f_{|[\M]^k})^p+\Lip_j(g_{|[\M]^k})^p \leq \Lip_j(h)^p + \varepsilon \]
for every $1 \leq j \leq k$.
\end{lemma}

\begin{proof}
Let $\varepsilon > 0$ and $h=(f,g) : ([\N]^k,\dH) \to X$ a Lipschitz map. \\
Let $\alpha_1=\inf\limits_{ \M_1 \in [\N]^{\omega}} \sup\limits_{(\overline{n},\overline{m}) \in H_1(\M_1)} \|f(\overline{n})-f(\overline{m})\|$. There exists $\M_1 \in [\M]^{\omega}$ so that $\|f(\overline{n})-f(\overline{m})\|^p \leq \alpha_1^p+ \frac{\varepsilon}{2}$ for every $(\overline{n},\overline{m}) \in H_1(\M_1)$. \\
Let $\beta_1=\inf\limits_{ \M_1' \in [\M_1]^{\omega}} \sup\limits_{(\overline{n},\overline{m}) \in H_1(\M_1')} \|g(\overline{n})-g(\overline{m})\|$. There exists $\M_1' \in [\M_1]^{\omega}$ so that $\|g(\overline{n})-g(\overline{m})\|^p \leq \beta_1^p+ \frac{\varepsilon}{2}$ for every $(\overline{n},\overline{m}) \in H_1(\M_1')$. \\
\begin{center} $\vdots$ \end{center} 
Let $\alpha_k=\inf\limits_{ \M_k \in [\M_{k-1}']^{\omega}} \sup\limits_{(\overline{n},\overline{m}) \in H_k(\M_k)} \|f(\overline{n})-f(\overline{m})\|$. There exists $\M_k \in [\M_{k-1}']^{\omega}$ so that $\|f(\overline{n})-f(\overline{m})\|^p \leq \alpha_k^p+ \frac{\varepsilon}{2}$ for every $(\overline{n},\overline{m}) \in H_k(\M_k)$. \\
Let $\beta_k=\inf\limits_{ \M_k' \in [\M_k]^{\omega}} \sup\limits_{(\overline{n},\overline{m}) \in H_k(\M_k')} \|g(\overline{n})-g(\overline{m})\|$. There exists $\M \in [\M_k]^{\omega}$ so that $\|g(\overline{n})-g(\overline{m})\|^p \leq \beta_k^p+ \frac{\varepsilon}{2}$ for every $(\overline{n},\overline{m}) \in H_k(\M)$. \\
Let us show that 
\[ \Lip_j(f_{|[\M]^k})^p+\Lip_j(g_{|[\M]^k})^p \leq \Lip_j(h)^p + \varepsilon \] 
for every $1 \leq j \leq k$. It is enough to show that $\alpha_j^p+\beta_j^p \leq \Lip_j(h)^p$. \\
For that, assume that it is not the case. Then, there exists $\eta > 0$ so that $(\alpha_j-\eta)^p + (\beta_j-\eta)^p > \Lip_j(h)^p$. If there exists $(\overline{n}, \overline{m}) \in H_j(\M)$ such that $\|f(\overline{n})-f(\overline{m})\| \geq \alpha_j- \eta$ and $\|g(\overline{n})-g(\overline{m})\| \geq \beta_j- \eta$, then $\| h(\overline{n})-h(\overline{m})\| > \Lip_j(h)$, which is impossible. So $\|f(\overline{n})-f(\overline{m})\| \leq \alpha_j- \eta$ or $\|g(\overline{n})-g(\overline{m})\| \leq \beta_j- \eta$ for all $(\overline{n}, \overline{m}) \in H_j(\M)$. Now we note that $H_j(\M)$ can be identified with $[\M]^{k+1}$ so, by Theorem \ref{ramsey}, we get $\M' \in [\M]^{\omega}$ such that $\|f(\overline{n})-f(\overline{m})\| \leq \alpha_j- \eta$ for all $(\overline{n}, \overline{m}) \in H_j(\M')$ or $\|g(\overline{n})-g(\overline{m})\| \leq \beta_j- \eta$ for all $(\overline{n}, \overline{m}) \in H_j(\M')$. This contradicts the definition of $\alpha_j$ or $\beta_j$ and finishes the proof.
\end{proof}

We can now prove our main result. Before doing so, we want to recall some facts about Kalton and Randrianarivony's result from \cite{KR}. If $f : [\N]^k \to \ell_p$ is a bounded map, they proved the existence of $u \in \ell_p$ so that, for every $\ee>0$, there exists $\M \in [\N]^\omega$ so that 
\[ \forall \nb \in [\M]^k, \ \|f(\nb)-u\| \leq \Big( \sum_{j=1}^k \Lip_j(f)^p \Big)^{1/p}+\ee . \]
By looking carefully at their proof, one can note more precisely that the $u$ is given by $u=\omega-\lim\limits_{n_1 \in \M \to \infty} \ldots \lim\limits_{n_k \in \M \to \infty} f(\nb)$ where $\M \in [\N]^\omega$ is such that the limits exist.

\begin{proof}[Proof of Theorem \ref{thm somme}]
Let $\varepsilon > 0$, $k \in \N$, $f : ([\N]^k,\dH) \to X$ a Lipschitz function. Let $\varepsilon''>0$ small enough so that 
\[ (\lambda+\varepsilon'')^p \sum_{j=1}^k \Lip_j(f)^p + k(\lambda+\varepsilon'')^p \varepsilon'' +\varepsilon'' \leq (\lambda+\varepsilon)^p \sum_{j=1}^k \Lip_j(f)^p \]
and $\varepsilon'>0$ small enough so that
\[ \left[ 2\left(\sum_{j=1}^k \Lip_j(f)^p \right)^{\frac{1}{p}}+ 4\varepsilon' \right]^p \leq 2^p \sum_{j=1}^k \Lip_j(f)^p +\varepsilon'' \]
The well-defined map \[ \phi : \left\lbrace \begin{array}{lll}
X & \to & \ell_p \\
(x_n)_{n \in \N} & \mapsto & (\|x_n\|)_{n \in \N}
\end{array} \right. \] 
satisfies $\Lip(\phi) \leq 1$ and $\| \phi(x)\|=\|x\|$ for all $x \in X$, thus
\[ \sup\limits_{(\overline{n},\overline{m}) \in H_j(\N)} \| \phi \circ f(\overline{n})-\phi \circ f (\overline{m}) \| \leq \Lip_j(f) \]
for every $j \in \{1, \cdots, k\}$. From Kalton-Randrianarivony's Theorem \cite{KR}, we get $u \in \ell_p$ and $\M'' \in [\N]^{\omega}$ such that 
\[ \| \phi \circ f(\overline{n})-u \| \leq \left( \sum_{j=1}^k \Lip_j(f)^p \right)^{\frac{1}{p}}+ \varepsilon' \]
for all $\overline{n} \in [\M'']^k$. Let $N \in \N$ be such that $\sum_{k=N+1}^{\infty} |u_k|^p \leq \varepsilon'^p$. Let us denote by $P_N$ the projection from $\ell_p$ onto $\vspan\{e_i,1 \leq i \leq N\}$ and by $\Pi_N$ the projection from $X$ onto $\left( \sum_{i=1}^N X_i \right)_{\ell_p}$. According to the lemma, there exists $\tilde{\M} \in [\M'']^\omega$ such that 
\[ \forall 1 \leq j \leq k, \ \Lip_j(\Pi_N \circ f_{|[\tilde{\M}]^k})^p+\Lip_j((I-\Pi_N) \circ f_{|[\tilde{\M}]^k})^p \leq \Lip_j(f)^p + \varepsilon'' \]
If we denote by $v=(I-P_N)(u)$, by the remark preceding this proof, there exists $\M' \in [\tilde{\M}]^\omega$ such that
\[ \| \phi \circ (I-\Pi_N) \circ f(\nb)- v\|=\|(I-P_N) \circ \phi \circ f(\nb)-v\| \leq \left( \sum_{j=1}^k \Lip_j((I-\Pi_N)\circ f_{|[\tilde{\M}]^k})^p \right)^{\frac{1}{p}}+ \varepsilon' \]
for all $\overline{n} \in [\M']^k$. According to Proposition 3.1 from \cite{fov}, there exists $\M \in [\M']^\omega$ such that 
\[ \| \Pi_N \circ f(\overline{n})-\Pi_N \circ f(\overline{m}) \| \leq (\lambda+\varepsilon'') \left( \sum_{j=1}^k \Lip_j(\Pi_N \circ f_{|[\tilde{\M}]^k})^p \right)^{\frac{1}{p}} \]
for all $\nb, \mb \in [\M]^k$ (resp. $(\overline{n}, \overline{m}) \in I_k(\M)$). Then, for all $\nb, \mb \in [\M]^k$ (resp. $(\overline{n}, \overline{m}) \in I_k(\M)$), we have
\begin{align*}
\|f(\overline{n})-f(\overline{m})\|^p & =\| \Pi_N(f(\overline{n})-f(\overline{m})) \|^p + \| (I-\Pi_N) \circ f(\overline{n})- (I-\Pi_N) \circ f(\overline{m}) \|^p \\
& \leq \| \Pi_N(f(\overline{n})-f(\overline{m})) \|^p + [\| (I-\Pi_N) \circ f(\overline{n})\|  + \| (I-\Pi_N) \circ f(\overline{m}) \|]^p \\
&= \| \Pi_N(f(\overline{n})-f(\overline{m})) \|^p + [\| \phi \circ (I-\Pi_N) \circ f(\overline{n})\|  + \| \phi \circ (I-\Pi_N) \circ f(\overline{m}) \|]^p \\
& \leq \| \Pi_N(f(\overline{n})-f(\overline{m})) \|^p + [\| \phi \circ (I-\Pi_N) \circ f(\overline{n})-v\|  + \\
& \hspace{5mm} \| \phi \circ (I-\Pi_N) \circ f(\overline{m})-v \|+2\|v\|]^p \\
& \leq (\lambda+\varepsilon'')^p \sum_{j=1}^k \Lip_j(\Pi_N \circ f_{|[\tilde{\M}]^k})^p + \\
& \hspace{5mm} \left[ 2\left(\sum_{j=1}^k \Lip_j((I-\Pi_N)\circ f_{|[\tilde{\M}]^k})^p \right)^{\frac{1}{p}}+ 4\varepsilon' \right]^p \\
& \leq (\lambda+\varepsilon'')^p \sum_{j=1}^k \Lip_j(\Pi_N \circ f_{|[\tilde{\M}]^k})^p + 2^p \sum_{j=1}^k \Lip_j((I-\Pi_N)\circ f_{|[\tilde{\M}]^k})^p +\varepsilon'' \\
& \text{(since $x \in \R^+ \mapsto (x+4\varepsilon')^p-x^p$ is non-decreasing)} \\
& \leq (\lambda+\varepsilon'')^p \sum_{j=1}^k [\Lip_j(\Pi_N \circ f_{|[\tilde{\M}]^k})^p +  \Lip_j((I-\Pi_N)\circ f_{|[\tilde{\M}]^k})^p] + \varepsilon''
\end{align*}
since $\lambda \geq 2$. \\
Thus, for all $(\overline{n}, \overline{m}) \in I_k(\M)$, we get
\[ \|f(\nb)-f(\mb)\|^p \leq (\lambda+\varepsilon'')^p \sum_{j=1}^k \Lip_j(f)^p + k(\lambda+\varepsilon'')^p \varepsilon'' +\varepsilon'' \leq (\lambda+\varepsilon)^p \sum_{j=1}^k \Lip_j(f)^p \]
by choice of $\varepsilon''$.
\end{proof}


\section{Applications}

\subsection{A counterexample}

In this subsection, we answer by the negative the following question: is a Banach space with some property HIC$_{p,d}$, $p \in (1, \infty)$, necessarily AUSable? We first start with an elementary lemma, which strengthens the intuition that the concentration properties are asymptotic ones.

\begin{lemma} \label{lemma R}
Let $\lambda>0$, $p \in (1, \infty)$ and $X$ be a Banach space with a finite codimensional subspace $Y$ that has property $\lambda$-HFC$_{p,d}$ (resp. $\lambda$-HIC$_{p,d}$). Then $X$ has $(\lambda+\varepsilon)$-HFC$_{p,d}$ (resp. $(\lambda+\varepsilon)$-HIC$_{p,d}$) for every $\varepsilon>0$.
\end{lemma}

\begin{proof}
Let $\varepsilon>0$, $f : ([\N]^k,\dH) \to X$ be a non-constant Lipschitz map and $P : X \to Y$ be a bounded linear projection. Let $\varepsilon'>0$ so that 
\[ \lambda \Big( \sum_{j=1}^k (\Lip_j(f)+\varepsilon')^p \Big)^{1/p}+\ee' \leq (\lambda+\ee) \Big( \sum_{j=1}^k \Lip_j(f)^p \Big)^{1/p} . \]
It follows from Ramsey’s theorem \ref{ramsey} and the norm compactness of bounded sets in finite dimensional spaces that there exists $\M \in [\N]^\omega$ such that
\[ \forall \nb, \mb \in [\M]^k, \ \|(I-P)(f(\nb)-f(\mb))\| \leq \varepsilon' . \]
Therefore $P \circ f : [\N]^k \to Y$ satisfies $\Lip_j(P \circ f) \leq \Lip_j(f)+\varepsilon'$ for every $j \in \{1, \ldots, k \}$. Applying our hypothesis on $Y$, we get $\M' \in [\M]^\omega$ so that
\[ \|P \circ f(\nb)-P \circ f(\mb) \| \leq \lambda \Big( \sum_{j=1}^k (\Lip_j(f)+\varepsilon')^p \Big)^{1/p} \]
for all $\nb, \mb \in [\M']^k$ (resp. $(\nb,\mb) \in I_k(\M')$). Then, by choice of $\varepsilon'$,
\[ \|f(\nb)-f(\mb)\| \leq \|P \circ f(\nb)-P \circ f(\mb) \| + \|(I-P)(f(\nb)-f(\mb))\| \leq (\lambda+\ee) \Big( \sum_{j=1}^k \Lip_j(f)^p \Big)^{1/p} \]
for all $\nb, \mb \in [\M']^k$ (resp. $(\nb,\mb) \in I_k(\M')$).
\end{proof}

Let $p \in (1, \infty)$. We will define spaces by transfinite induction. Set $X_p^0=\R \oplus_1 \ell_p$, $X_p^{\alpha+1}=\R \oplus_1 \ell_p(X_p^\alpha)$ for every ordinal $\alpha$ and $X_p^\alpha= \R \oplus_1 \Big(\sum_{\beta<\alpha} X_p^\beta)_{\ell_p}$ if $\alpha$ is a limit ordinal. They satisfy $Sz(X_p^\alpha) > \alpha$ for every $\alpha<\omega_1$. Indeed, one can show with a transfinite induction that $(1,0) \in l_1^\alpha(B_{(X^\alpha_p)^*})$ for every ordinal $\alpha$. Therefore, we can deduce from Theorem \ref{thm somme} the following corollary, that answers a question from \cite{fov}.

\begin{corollary} \label{cor AUS}
For every ordinal $\alpha<\omega_1$ and $p \in (1, \infty)$, there exists a Banach space with property HFC$_{p,d}$ and Szlenk index bigger than $\alpha$. By Theorem \ref{thm Sz}, property HFC$_{p,d}$ for some $p \in (1, \infty)$ does not imply being AUSable.
\end{corollary}

In particular, for every ordinal $\alpha<\omega_1$, one gets a first example of a Banach space with Szlenk index bigger than $\alpha$ that does not equi-Lipschitz contain the Hamming graphs (nor the symmetric graphs). We can also deduce from Corollary \ref{cor AUS} that there is no separable Banach space that contains an isomorphic copy of every separable Banach space $X$ with property HFC$_{p,d}$. Let us finish with subsection with a last remark. In the local setting, a Banach space $X$ has non-trivial Rademacher type if and only if $\ell_1$ is not finitely representable in $X$ if and only if it does not equi-Lipschitz contain the Hamming cubes \cite{BMW} \cite{Pisier_sur_les_espaces}. In \cite{DKC}, an asymptotic version of Rademacher type $p$, $p \in (1, \infty)$ is introduced and the following is proved: a banach space has non-trivial asymptotic Rademacher type if and only if $\ell_1$ is not asymptotically finitely representable in $X$, \ie if $\ell_1$ is not in its asymptotic structure. To ease the reading, we will not recall the definition of asymptotic structure but we will mention that $\ell_1$ happens to be in the asymptotic structure of $X_p^\omega$ (see \cite{BLMSJIMJ2021}). Therefore, with the definition of asymptotic Rademacher type from \cite{DKC}, contrary to what happens in the local setting, there exist Banach spaces with trivial asymptotic Rademacher type that do not equi-Lipschitz contain the Hamming graphs.

\subsection{A complexity result}

Let us denote by $C(\Delta)$ the space of continuous real-valued functions on the Cantor set $\Delta$. Since this space is universal for the class of separable Banach spaces, we can code that class by $\SB=\{ X \subset C(\Delta), X \text{ is a closed linear subspace of } C(\Delta)\}$. As $C(\Delta)$ is a Polish space, we can endow the set $\mathcal{F} (C(\Delta))$ of its closed non-empty subsets with the Effros-Borel structure (see Chapter 2 from \cite{Dodos} or \cite{Bos}). Hence we can talk about Borel and (completely) co-analytic classes of separable Banach spaces. We refer to the textbook \cite{Kechris} for the definitions and a complete exposition of descriptive set theory. We will just recall that an \textit{analytic set} is a continuous image of a Polish space into another Polish space, a \textit{co-analytic set} is a set whose complement is analytic and an analytic (resp. co-analytic) subset $A$ of a standard Borel space $X$ is said to be \textit{completely analytic} (resp. \textit{completely co-analytic}) if for each standard Borel space $Y$ and each $B \subset Y$ analytic (resp. co-analytic), one can find a Borel function $f : X \to Y$ so that $f^{-1}(B)=A$. The following well-known lemma will be used in the proof of Proposition \ref{prop complexity}.

\begin{lemma}[Kutarowski-Ryll-Rardzewski selection principle (see Theorem (12.13) \cite{Kechris})] \label{lemme KRR}
Let $X$ be a Polish space, $\mathcal{F}(X)$ the set of its non-empty closed sets. \\ 
There exists a sequence of Borel functions $(d_n)_{n \in \N} : \mathcal{F}(X) \to X$ such that $(d_n(F))_{n \in \N}$ is dense in $F$, for every non-empty closed subset $F \subset X$.
\end{lemma}

As a consequence of Corollary \ref{cor AUS}, since Bossard \cite{Bos} proved that the map $X \mapsto Sz(X)$ is a co-analytic rank on the set of spaces with separable dual, the set of spaces with separable dual satisfying HFC$_{p,d}$, $p \in (1, \infty)$ is not Borel (see for example Theorem A.2 \cite{Dodos}). In fact, we even have:

\begin{prop} \label{prop complexity}
Let $p \in (1, \infty)$. The set of separable Banach spaces satisfying HFC$_{p,d}$ is completely co-analytic. 
\end{prop}

Before proving this proposition, let us recall some terminology. We denote by $\Tr$ the set of all trees over $\N$, where a tree over $\N$ is a set $T$ such that $\varnothing \in T$, $T \subset \bigcup_{k=1}^\infty [\N]^k$ and $(n_1, \cdots, n_k) \in T$ as soon as $(n_1, \cdots, n_{k+1}) \in T$ for some $n_{k+1}>n_k$. A natural partial order $\preceq$ can be defined on a tree $T \in \Tr$ by setting $\varnothing \preceq \nb$ for every $\nb \in T$ and $(n_1, \cdots, n_j) \preceq (m_1, \cdots, m_k)$ if $k \geq j$ and $m_i=n_i$ for every $1 \leq i \leq j$. A tree is said well-founded if it does not contain an infinite increasing sequence, ill-founded if it does. Finally, a linearly ordered subset $I \subset T$ with respect to $\preceq$ is called a segment if $T$ and two segments $I,J$ of $T$ are said to be incomparable if we don't have $\nb \preceq \mb$ or $\mb \preceq \nb$ for any $\nb \in I_1$, $\mb \in I_2$.

\begin{proof}[Proof of Proposition \ref{prop complexity}]
First, let us prove that having HFC$_{p,d}$ is a co-analytic condition. We will use Lemma \ref{lemme KRR}. By definition, a Banach space $X$ does not have property HFC$_{p,d}$ if and only if for every $\lambda \in \N$, there exists $C \in \N$, $k \in \N$ and $f : [\N]^k \to \N$ so that
\[ \|d_{f(\nb)}(CB_X)-d_{f(\mb)}(CB_X)\| \geq \lambda \sum_{j=1}^k \Lip_j(d_{f(\cdot)}(CB_X)) \]
for all $\nb < \mb \in [\N]^k$. Since the map $\left\lbrace \begin{array}{lll}
SB \to \mathcal{F}(C(\Delta)) \\
X \mapsto CB_X \end{array} \right.$ is Borel, not having HFC$_{p,d}$ is an analytic condition and therefore, having HFC$_{p,d}$ is a co-analytic condition. Now, to prove the proposition, it is enough to find a Borel map $\varphi$ from the set of trees $\Tr$ into the set of separable Banach spaces so that, for every tree $T \in \Tr$, $T$ is well-founded if and only if $\varphi(T)$ has HFC$_{p,d}$. To do so, let $q \in (1, p)$ and let us denote by $(e_n)_{n \in \N}$ the canonical basis of $\ell_q$. If $T \in \Tr$, $x=(x_{\nb})_{nb \in T} \in c_{00}(T)$ and $I$ is a segment of $T$, we let $x_{|I}= \sum_{\nb \in I} x(\nb) e_{\max(\nb)} \in \ell_q$. Now, for each $T \in \Tr$ and $x \in c_{00}(T)$, let
\[ \|x\|_T=\sup \Big\{ \Big( \sum_{i=1}^n \|x_{|I_i}\|_{\ell_q}^p \Big)^{1/p}, \ I_1, \cdots, I_n \text{ incomparable segments of } T \Big\}  \]
and denote by $X_T$ the completion of $c_{00}(T)$ under the norm $\|.\|_T$. By a transfinite induction on the order of $T$ and Theorem \ref{thm somme}, we get that $X_T$ has property HFC$_{p,d}$ for every well-founded tree $T \in \Tr$ (see \cite{Braga3} for a similar transfinite induction). Since $X_T$ contains an isometric copy of $\ell_q$ for every ill-founded tree $T$ and $\ell_q$ does not have property HFC$_{p,d}$, the Borel function $\varphi : T \mapsto X_T$ is so that $T$ is well-founded if and only if $\varphi(T)$ has HFC$_{p,d}$.
\end{proof}

\begin{rmk}
Likewise, the set of separable Banach spaces satisfying HIC$_{p,d}$ is completely co-analytic.
\end{rmk}

\begin{corollary}
Let $p \in (1, \infty)$. There is no separable Banach space $Z$ with property HFC$_{p,d}$ (resp. HIC$_{p,d}$) such that every separable Banach space with HFC$_{p,d}$ (resp. HIC$_{p,d}$) can be coarse-Lipschitz embedded into $Z$.
\end{corollary}

\begin{proof}
By contradiction, let us assume that we can find such a separable space $Z$ with HFC$_{p,d}$. If we denote by $\operatorname{SB}$ the class of separable Banach spaces, by $\mathcal{HFC}_{p,d}$ the set of separable Banach spaces with property HFC$_{p,d}$ and $\operatorname{CLE}_Z=\{ Y \in \operatorname{SB} ; Y \text{ coarse Lipschitz embeds into } Z \}$, then $\mathcal{HFC}_{p,d}=\operatorname{CLE}_Z$ (let us recall that property HFC$_{p,d}$ is stable under coarse Lipschitz embedding). It is easily checked that $\operatorname{CLE}_Z$ is analytic (see the proof of Theorem 1.7 - Section 7.1 in \cite{Braga2}), which contradicts Proposition \ref{prop complexity}. The same can be done with HIC$_{p,d}$.
\end{proof}

\section{Final remarks and open problems}

In view of Corollary \ref{cor AUS}, the following question seems natural.

\begin{pb}
If a Banach space $X$ has property HIC$_{p,d}$ for some $p \in (1, \infty)$, is it Asplund?
\end{pb}

We will finish this paper by recalling the important Problem 2 from \cite{GLZ2014}.

\begin{pb}
If a Banach space $X$ coarse Lipschitz embeds into a Banach space $Y$ that is reflexive and AUS, does it follow that $X$ is AUSable?
\end{pb}

It follows from Corollary \ref{cor AUS} that concentration properties are not the right tool to answer this question. If the answer were to be negative, a counterexample could come from the space $X_2^\omega$.

\begin{thank}
The author would like to thank Gilles Lancien for very useful conversations and valuable comments. The author also thanks the referee for her/his suggestions that helped improve this article.
\end{thank}

\bibliographystyle{plain}
\bibliography{biblio}

\end{document}